\documentclass{article}

\newcommand{\Debug}{0}

\usepackage{amsmath, amsthm, amssymb, amstext, epsf,amsfonts}
\usepackage{enumerate}
\usepackage{graphicx}
\usepackage[applemac]{inputenc}

\usepackage[percent]{overpic}

\theoremstyle{plain}

\newtheorem{thm}{Theorem}[section]
\newtheorem{lem}[thm]{Lemma}
\newtheorem{cor}[thm]{Corollary}
\newtheorem{prop}[thm]{Proposition}
\newtheorem{rem}[thm]{Remark}

\newtheorem{ques}[thm]{{Question}}

\theoremstyle{definition}
\newtheorem{defi}[thm]{Definition}

\newcommand{\Cg}{Cayley graph}

\newcommand{\R}{\ensuremath{\mathbb{R}}}

\newcommand{\Z}{\ensuremath{\mathbb{Z}}}

\newcommand{\BS}{\ensuremath{\mathbb S}}

\newcommand{\sm}{\ensuremath{\setminus}}

\newcommand{\sub}{\subseteq}

\newcommand{\comment}[1]{}

\newcommand{\nat}{{\mathbb N}}
\newcommand{\real}{{\mathbb R}}

\newcommand{\DF}{\ensuremath{\mathcal D}}

\newcommand{\FF}{\ensuremath{\mathcal F}}
\newcommand{\GF}{\ensuremath{\mathcal G}}

\newcommand{\IF}{\ensuremath{\mathcal I}}

\newcommand{\PF}{\ensuremath{\mathcal P}}

\newcommand{\RF}{\ensuremath{\mathcal R}}
\newcommand{\SF}{\ensuremath{\mathcal S}}

\newcommand{\VF}{\ensuremath{\mathcal V}}
\newcommand{\WF}{\ensuremath{\mathcal W}}

\newcommand{\splpr}{special \plpr}
\newcommand{\Splpr}{Special \plpr}
\newcommand{\gplpr}{general \plpr}

\newcommand{\utr}{up to reflection}

\ifnum \Debug = 1 
	\def\?#1{\vadjust{\vbox to 0pt{\vss\vskip-8pt\leftline{%
     \llap{\hbox{\vbox{\pretolerance=-1
     \doublehyphendemerits=0\finalhyphendemerits=0
     \hsize16truemm\tolerance=10000\small
     \lineskip=0pt\lineskiplimit=0pt
     \rightskip=0pt plus16truemm\baselineskip8pt\noindent
     \hskip0pt        
     #1\endgraf}\hskip7truemm}}}\vss}}}
\else \newcommand{\?}[1]{} \fi

\begin{document}

\title{The planar Cayley graphs are effectively enumerable I: consistently planar graphs}
\author{Agelos Georgakopoulos\thanks{Supported by EPSRC grant EP/L002787/1, and by the European Research Council (ERC) under the European Union’s Horizon 2020 research and innovation programme (grant agreement No 639046). The first author would like to thank the Isaac Newton Institute for Mathematical Sciences, Cambridge, for support and hospitality during the programme `Random Geometry' where work on this paper was undertaken.}
\medskip 
\\
  {Mathematics Institute}\\
 {University of Warwick}\\
  {CV4 7AL, UK}\\
\and
 Matthias Hamann\thanks{Both authors have been supported by FWF grant P-19115-N18.} 
  \medskip \\
{Alfr\'ed R\'enyi Institute of Mathematics}\\
{Hungarian Academy of Sciences}\\
{Budapest, Hungary}\\
\\
}
\date{\today}
\maketitle

\newcommand{\labtequ}[2]{ \begin{equation} \label{#1} 	\begin{minipage}[c]{0.9\textwidth} \centering #2 \end{minipage} \ignorespacesafterend \end{equation} } 
\newcommand{\showFig}[2]{
   \begin{figure}[htbp]
   \centering
   \noindent
\includegraphics{#1.eps}
   \caption{\small #2}
   \label{#1}
   \end{figure}
}
\newcommand{\fig}[1]{Figure~{\ref{#1}}}

\newcommand{\cR}{\ensuremath{\mathcal{R}}}
\newcommand{\srev}{spin-reversing}
\newcommand{\plpr}{planar presentation}
\newcommand{\empr}{embedded presentation}

\newcommand{\Plpr}{Planar presentation}
\newcommand{\spre}{spin-preserving}
\newcommand{\Gam}{\ensuremath{\Gamma}}
\newcommand{\sig}{\ensuremath{\sigma}}
\newcommand{\g}{\ensuremath{G\ }}
\newcommand{\G}{\ensuremath{G}}
\renewcommand{\iff}{if and only if}
\newcommand{\fe}{for every}
\newcommand{\Fe}{For every}
\newcommand{\fea}{for each}
\newcommand{\Fea}{For each}
\newcommand{\st}{such that}
\newcommand{\sot}{so that}
\newcommand{\ti}{there is}
\newcommand{\ta}{there are}

\newcommand{\cns}{consistent}
\newcommand{\cnsem}{consistent embedding}
\newcommand{\vapf}{accumulation-free}
\newcommand{\dray}{double ray}
\newcommand{\cls}[1]{\overline{#1}}


\newcommand{\Lr}[1]{Lemma~\ref{#1}}
\newcommand{\Lrs}[1]{Lemmas~\ref{#1}}
\newcommand{\Tr}[1]{Theorem~\ref{#1}}
\newcommand{\Trs}[1]{Theorems~\ref{#1}}
\newcommand{\Sr}[1]{Section~\ref{#1}}
\newcommand{\Srs}[1]{Sections~\ref{#1}}
\newcommand{\Prr}[1]{Pro\-position~\ref{#1}}
\newcommand{\Prb}[1]{Problem~\ref{#1}}
\newcommand{\Cr}[1]{Corollary~\ref{#1}}
\newcommand{\Cnr}[1]{Con\-jecture~\ref{#1}}
\newcommand{\Or}[1]{Observation~\ref{#1}}
\newcommand{\Er}[1]{Example~\ref{#1}}
\newcommand{\Dr}[1]{De\-fi\-nition~\ref{#1}}

\newcommand{\kreis}[1]{\mathaccent"7017\relax #1}

\newcommand{\ttt}{\ensuremath{\mathbb T}}
\newcommand{\sydi}{\triangle}
\newcommand{\Gc}{\ensuremath{\overline{G}}}
\newcommand{\cells}{\ensuremath{G^2}}

\begin{abstract}
We obtain an effective enumeration of the family of finitely generated groups admitting a faithful, properly discontinuous action on some 2-manifold contained in the sphere. This is achieved by introducing a type of group presentation capturing exactly these groups.

Extending this in a companion paper, we 
find group presentations capturing the planar finitely generated  \Cg s.
Thus we obtain an effective enumeration of these Cayley graphs, yielding in particular an affirmative answer to a question of Droms et al. 

\end{abstract}

\section{Introduction}

\subsection{Overview} \label{sec overv}
Groups which act discretely on the real plane $\R^2$ by homeomorphisms, called \emph{planar discontinuous groups}, are a classical topic the study of which goes back at least as far as Poincar\'e \cite{PoiThe}, and are now fully classified and considered to be well understood. 
The finite ones were classified by
Maschke \cite{Maschke}, and important contributions to the infinite case where made by Wilkie \cite{wilNon} and Macbeath \cite{macCla}. 
See \cite[Prop.~III.~5.~4]{LyndonSchupp} or \cite{TucFin,ZVC} for more. These groups are closely related to surface groups, which have influenced most of combinatorial group theory \cite{AckFiRoSur}.

Planar discontinuous groups coincide with the groups admitting a planar modified\footnote{`modified' in the sense that redundant 2-cells are removed; see \cite{LyndonSchupp} for details.} Cayley complex, and they are Fuchsian when they do not contain orientation-reversing elements \cite{AckFiRoSur}. These Cayley complexes correspond exactly to the \Cg s that can be embedded into $\R^2$ without accumulation points of vertices.

Planar \Cg s that can only be embedded in $\R^2$ with accumulation points of vertices, and their groups, are a bit harder to understand, and still a topic of ongoing research  \cite{ArzChe,DroInf,DrSeSeCon,dunPla,cayley3,cay2con,mohTre}. An important fact is that they are finitely presented and hence accessible \cite{DroInf,dunPla,Dun}.
Dunwoody \cite[Theorem~3.8]{dunPla} uses this to prove that such a group, or a subgroup of index two of its, is a fundamental group of a graph of groups in which each vertex group is either a planar discontinuous group or a free product of finitely many cyclic groups and all edge groups are finite cyclic groups (possibly trivial).


\medskip
In this paper, along with the companion paper \cite{planarPresII}\footnote{An older combined version of these two papers is available at  {\small http://arxiv.org/abs/1506.03361v2}.}
we extend the aforementioned correspondence between planar Cayley complexes and accumulation-free planar \Cg s by relaxing the notion of planarity of a Cayley complex in such a way that it corresponds exactly to its \Cg\ being planar. Our new notion of `\emph{almost planarity}' of a Cayley complex can be recognised algorithmically (as discussed in \cite{planarPresII}). In view of the Adjan-Rabin theorem \cite{Adjan,Rabin}, this is a rather rare case of a decidable, geometric property of Cayley complexes. 

The key to this is a concept of \emph{planar group presentation} we introduce; this is a type of group presentation that guarantees the planarity of the corresponding \Cg, and conversely, we show that every planar, finitely generated \Cg\ admits such a group presentation. 

In this paper we concentrate on the subfamily of such \Cg s $G$ that admit an embedding in the sphere $\BS^2$ such that the action of the corresponding group on $G$ preserves face-boundaries. Such an embedding is called {\em consistent}. We show in a follow-up paper in preparation that the groups having such \Cg s are exactly those finitely generated groups admitting a faithful, properly discontinuous action by homeomorphisms on a 2-manifold contained in  $\BS^2$. We will define a type of group presentation characterising these groups, obtaining in particular an effective enumeration of this family of groups. 


\subsection{Results}\label{subsec_results}
The \emph{Cayley complex} $X$ corresponding to a group presentation $\PF=\left< \SF \mid \mathcal \RF \right>$ is the 2-complex obtained from the \Cg\ $G$ of $\PF$ by glueing a 2-cell along each closed walk of $G$ induced by a relator $R\in \RF$. We say that $X$ is \emph{almost planar}, if it admits a map $\rho: X \to \R^2$ \st\ the 2-simplices of $X$ are \emph{nested} in the following sense. We say that  two 2-simplices of $X$ are nested, if the images of their interiors are either disjoint, or one is contained in the other, or their intersection is the image of a 2-cell bounded by two parallel edges corresponding to an involution $s\in \SF$.\footnote{The third option can be dropped by considering the \emph{modified} Cayley complex in the sense of \cite{LyndonSchupp}, i.e.\ by representing involutions in $\SF$ by single, undirected edges.} We call the presentation $\PF$ a \emph{\plpr} if its Cayley complex is almost planar. Our first result is

\begin{prop} \label{thmPlpr}
Every planar, finitely generated \Cg\ admits a \plpr.
\end{prop}

The main idea behind this is that if two relators in a presentation induce cycles whose interiors overlap but are not nested, then we could replace a subword of one relator by a subword of the other to produce an equivalent presentation with less overlapping; our proof that a presentation with no such overlaps exists is based on the machinery of Dunwoody cuts, cf.~\cite{dicks_dunw}.

In fact, we prove something much stronger than \Prr{thmPlpr}. We introduce a specific type of \plpr, called a \emph{\gplpr}, and show that every planar, finitely generated \Cg\ admits such a
presentation and, conversely, every \gplpr\ has a planar \Cg\ (Theorem 4.7.\ in the companion paper \cite{planarPresII}). This converse is the hardest result of our work. Moreover, as \gplpr s can be recognised algorithmically, we deduce 

\begin{cor}[\cite{planarPresII}] \label{thmEffEn}
The set of planar, finitely generated \Cg s can be effectively enumerated.
\end{cor}
By an effective enumeration of an infinite set $\GF$ we mean a computer program that outputs elements of this set and nothing else, and every element of $\GF$ is in the output after some finite time (repetitions are allowed).

This implies a positive answer to a question of Droms et al.\ \cite{DroInf,DrSeSeCon},  namely whether the groups admitting a planar, finitely generated \Cg\ (called \emph{planar groups}) can be effectively enumerated. This question is motivated by the fact that, as a consequence of the Adjan-Rabin theorem \cite{Adjan,Rabin}, planar groups cannot be recognised by an algorithm (taking a presentation as input), and by the fact that planar discontinuous groups have been effectively enumerated~\cite{DrSeSeCon}. M.~Dunwoody (private communication) informs us that the fact that the planar groups  can be effectively enumerated should also follow from his result \cite[Theorem~3.8]{dunPla} mentioned above with a little bit of additional work (the main issue here is whether the `or a subgroup of index two' proviso can be dropped).

We remark that there is a huge variety of planar \Cg s: even the 3-regular ones form 37 infinite families \cite{cayley3,cay2con}. Moreover, as the same group can have many planar \Cg s, \Cr{thmEffEn} is stronger than saying that planar groups can be effectively enumerated. The aforementioned classification of the 3-regular planar \Cg s cost the first author about 100 pages of work \cite{cayley3,cay2con}. This task would have been significantly simplified if our results had been available at that time, using a computer aided search based on the algorithm behind \Cr{thmEffEn}, which is straightforward to implement.

Our proof method is essentially graph-theoretic, and does not appeal to the theory of planar discontinuous groups. We are optimistic that it can be employed in a wider setup including groups like the Baumslag-Solitar group that act on spaces that generalise the plane.

\comment{
\medskip
We remark in \Cr{corresf} that planar groups are residually finite. We therefore replace the question of Droms et al.\ with the following, which might be too bold
\begin{ques} \label{qurf}
Is every family of finitely presented, residually finite groups  which is closed under taking subgroups effectively enumerable?
\end{ques}
}

\subsection{\Plpr s}
We now sketch the fundamental concept of this paper, called a \splpr\ (\gplpr s are introduced in \cite{planarPresII}). Such presentations always exist for a 3-connected planar \Cg, or more generally, for a \Cg\ that can be embedded in the plane in such a way that its label-preserving automorphisms carry facial paths to facial paths.

We say that $\PF=\left< \SF \mid \mathcal \RF \right>$ is a \emph{\splpr}, if it can be endowed with a cyclic ordering $\sig$ ---from now on called a \emph{spin}--- of the symmetrization $\SF'=\{s,s^{-1} \mid s \in \SF\}$ of its generating set, with the following property.
Suppose $W_1=sUt$ and $W_2=s'Ut'$, where $s,s',t,t'\in \SF'$, are two words, each contained in some rotation of a relator in $\RF$  (possibly the same relator), where $U$ is any (possibly trivial) word with letters in $\SF'$. Then $\sig$ allows us to say whether paths induced by these words $W_1,W_2$ would \emph{cross} each other or not if we could embed the \Cg\ of $\PF$ in the plane in such a way that for every vertex the cyclic ordering of the labels of its incident edges we observe coincides with $\sig$. To make this more precise, we embed a tree consisting of a `middle' path $P$ with edges labelled by the letters in $U$, and two leaves attached at each endvertex of~$P$ labelled with $s,s',t,t'$ as in \fig{stu}, where the spin we use at each endvertex of~$P$ is the one induced by $\sig$ on the corresponding 3-element subset of $\SF'$. There are essentially two situations that can arise, both shown in that figure. Naturally, we say that $W_1,W_2$ \emph{cross} each other in the right-hand situation, and they do not in the left-hand one.

\showFig{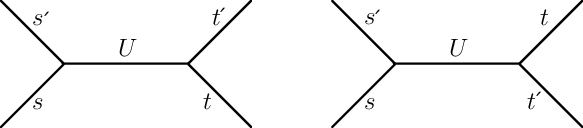}{The definition of \emph{crossing}; $W_1=sUt$ crosses $W_2=s'Ut'$ in the right, but not in the left.}

We then say that $\PF$ is a \splpr, if there is a spin $\sig$ on $\SF'$ \st\ no two words as above cross each other. Note that this is an abstract property of sets of words, and it is defined without reference to the \Cg\ of $\PF$; in fact, it can be checked algorithmically. The main essence of this paper is that this is enough to guarantee the planarity of the \Cg, and that a converse statement holds. 

This generalises an idea from \cite{vapf}, where it was shown that every planar discontinuous group admits a \splpr\ where every relator is \emph{facial}, i.e.\ it crosses no other word (where we consider words that are not necessarily among our relators).

Our actual definition of a \splpr, given in \Sr{secPlpr}, is in fact a bit more general than the above sketch. Consider for example the \Cg\ of the presentation $\left<a,b \mid a^n, b^2, aba^{-1}b\right>$. Its \Cg\ is a prism graph with an essentially unique embedding in $\R^2$. Note that the spin of half of its vertices is the reverse of the spin of other half. This is a general phenomenon: every 3-connected \Cg\ has an essentially unique embedding, and in that embedding all vertices have the same spin \emph{\utr}. However, for every generator $s$, either the two end-vertices of all edges labelled $s$ have the same spin, or they always have reverse spins. This yields a classification of generators into \emph{\spre} and \emph{\srev} ones, and our definition of a \splpr\ takes this into account; still, everything can be checked algorithmically.

The situation becomes much more complex if one wants to consider planar \Cg s that are not 3-connected. Such graphs do not always have an embedding with all vertices having the same spin \utr; an example is given in \cite{DrSeSeCon}. 
In order to capture such \Cg s we 
introduce \emph{\gplpr} in \cite{planarPresII}. 


\medskip
\Splpr s are important not only as a building block for \gplpr s, but also because of the following
\begin{thm} \label{thm cons}
A \Cg\ admits a consistent embedding \iff\ it admits a \splpr.
Thus the finitely generated \Cg s that admit a consistent embedding can be effectively enumerated.
\end{thm}
Here, we call an embedding $\rho:G \to \R^2$ \emph{consistent}, if every vertex has the same spin \utr, and every generator is either \spre\ or \srev\ in the above sense. Equivalently, $\rho$ is consistent, if the label-preserving automorphisms of $G$ carry every facial path with respect to $\rho$ to a facial path.

As mentioned in \Sr{sec overv}, the groups having \Cg s as in \Tr{thm cons} coincide with those finitely generated groups admitting a faithful, properly discontinuous action by homeomorphisms on a (topological) 2-manifold contained in  $\BS^2$, which is proved in a follow-up paper in preparation. (This 2-manifold can be assumed to be the sphere, the plane, the open cylinder, or the Cantor sphere, depending on whether the group has 0, 1, 2, or infinitely many ends, respectively.)
Thus \splpr s characterise these groups, yielding in particular an effective enumeration.

\medskip
This paper is structured as follows. After some general definitions, we introduce \splpr s in \Sr{3con}, and show that every 3-connected planar \Cg\ admits such a presentation (\Tr{thm_3con}). We extend this in \cite[Theorem~5.6]{planarPresII} by replacing the 3-connectedness condition with the weaker condition of admitting a consistent embedding. Next, we show that the \Cg\ of every \splpr\ is planar in \Sr{secplty} (\Tr{thmplanar}). Combining these two results yields the first sentence of \Tr{thm cons}. It is easy to see that \splpr s can be recognised algorithmically (see \cite{planarPresII}), which then implies the second sentence. The results of this paper are used in \cite{planarPresII} in order to obtain \Cr{thmEffEn} and other related results.


\section{Definitions}

\subsection{\Cg s and group presentations} \label{defpres}
We will follow the terminology of~\cite{diestelBook05} for graph-theoretical terms and that of~\cite{bogop} and~\cite{GroupsGraphsTrees} for group-theoretical ones. Let us recall the definitions most relevant for this paper.

A {\em group presentation} $\left< \SF \mid \mathcal \RF \right>$ consists of a set $\SF$ of distinct symbols, called the  {\em generators} and a set $\RF$ of words with letters in $\SF \cup \SF^{-1}$, where $\SF^{-1}$ is the set of symbols $\{s^{-1} \mid s \in \SF\}$, called  {\em relators}. Each such group presentation uniquely determines a group, namely the  quotient group $F_\SF/N$ of the (free) group $F_\SF$ of words with letters in $\SF \cup \SF^{-1}$ over the (normal) subgroup $N=N(\RF)$ generated by all conjugates of elements of $\RF$.

The \Cg\  $Cay(\PF)= Cay \left< \SF \mid \mathcal \RF \right>$  of  a group presentation $\PF= \left< \SF \mid \mathcal \RF \right>$
is an edge-coloured directed graph $\g= (V,E)$ constructed as follows. The vertex set of \g is the group $\Gam=F_\SF /N$ corresponding to $\PF$. The set of {\em colours} we will use is $\SF$.  For every $g\in \Gam, s\in \SF$ join $g$ to $gs$ by an edge coloured $s$ directed from $g$ to $gs$. Note that $\Gam$ acts on \g by multiplication on the left; more precisely, \fe\ $g\in \Gam$ the mapping from $V(G)$ to $V(G)$ defined by $x \mapsto gx$ is an {\em automorphism} of \G, that is, an automorphism of \g that preserves the colours and directions of the edges. In fact, \Gam\ is precisely the group of such automorphisms of \G. Any presentation of \Gam\ in which $\SF$ is the set of generators will also be called a presentation of $Cay(\PF)$.

Note that some elements of $\SF$ may represent the identity element of \Gam, and distinct elements  of $\SF$ may represent the same element of \Gam; therefore, $Cay(\PF)$ may contain loops and parallel edges of the same colour.

If $s\in \SF$ is an {\em involution}, i.e.\ $s^2=1$, then every vertex of \g is incident with a pair of parallel edges coloured $s$ (one in each direction). 
If $s^2$ is a relator in~$\RF$, we will follow the convention of replacing this pair of parallel edges by a single, undirected edge.
This convention is common in the literature \cite{LyndonSchupp}, and is convenient when  studying planar \Cg s.
\medskip

If \g is a \Cg, then we use $\Gam(G)$ to denote its group.

\medskip
If $R$ is any  (finite or infinite) word with letters in $\SF \cup \SF^{-1}$, and $g$ is a vertex of $G=Cay \left< \SF \mid \mathcal \RF \right>$, then starting from $g$ and following the edges corresponding to the letters in $R$ in order we obtain a walk $W$ in $G$. We then say that $W$ is {\em induced} by $R$ at $g$, and we will sometimes denote $W$ by $gR$; note that for a given $R$ \ta\ several walks in \g induced by $R$, one for each starting vertex $g\in V(G)$. 

Let $H_1(G)$ denote the first simplicial homology group of \g over $\Z$. 
We will use the following well-known fact which is easy to prove.
\begin{lem} \label{relcc}
Let $G= Cay \left< \SF \mid \RF \right>$ be a \Cg. Then the (closed) walks in \g induced by relators in $\RF$ generate $H_1(G)$. 
\end{lem}

\subsection{Graph-theoretical concepts}
Let $G=(V,E)$ be a connected graph fixed throughout this section. The set of neighbours of a vertex $x$ is denoted by {\em $N(x)$}.

Two paths in \g are {\em independent}, if they do not meet  at any vertex except perhaps at common endpoints. 
If $P$ is a path or cycle we will use $|P|$ to denote the number of vertices in $P$ and  $||P||$ to denote the number of edges of $P$. Let $xPy$ denote the subpath of $P$ between its vertices $x$ and $y$.
If $W$ is a walk, we denote by $\kreis{W}$ its subwalk consisting of~$W$ without its first and last edges.

A \emph{directed edge} of $G$ is an ordered pair $(x,y)$ such that $xy\in E(G)$. Thus any edge $xy=yx\in E$ corresponds to two directed edges. 


A {\em hinge} of \g is an edge $e=xy$ \st\ the removal of the pair of vertices $x,y$  disconnects \G. A hinge should not be confused with a {\em bridge}, which is an edge whose removal separates \g  although its endvertices are not removed.

\G\ is called {\em $k$-connected} if $G - X$ is connected for every set $X\subseteq V$ with $|X | < k$. Note that if \g is $k$-connected then it is also $(k-1)$-connected. The {\em connectivity $\kappa(G)$} of \g is the greatest integer $k$ such that \G\ is $k$-connected.

A $1$-way infinite path is called a {\em ray}. Two rays are {equivalent} if no finite set of vertices separates them eventually. The corresponding equivalence
classes of rays are the {\em ends} of $G$. A graph is {\em multi-ended} if it has more than one end. Note that given any two finitely generated presentations of the same group, the corresponding \Cg s have the same number of ends. Thus this number, which is known to be one of $0,1,2, \infty$, is an invariant of finitely generated groups.

A {\em \dray} is a directed $2$-way infinite path.

The set of all finite sums of (finite) cycles forms a vector space over $\mathbb F_2$, the \emph{cycle space} of~$G$.

\subsection{Embeddings in the plane} \label{secDem}	

An {\em embedding} of a graph \g will always mean a topological embedding of the corresponding 1-complex in the euclidean plane $\R^2$; in simpler words, an embedding is a drawing in the plane with no two edges crossing.

A {\em face} of an embedding $\rho\colon G \to \R^2$ is a component of $\R^2 \sm \rho(G)$. The {\em boundary} of a face $F$ is the set of vertices and edges of \g that are mapped by $\rho$ to the closure of $F$. The {\em size} of $F$ is the number of edges in its boundary. Note that if $F$ has finite size then its boundary is a cycle of \G.

A walk in \g is called {\em facial} with respect to~$\varphi$ if it is contained in the boundary of some face of~$\varphi$. 

Given an embedding $\varphi$ of a \Cg\ $G$ with generating set $\SF$, we consider \fe\ vertex $x$ of \g the embedding of the edges incident with $x$, and define the {\em spin} of $x$ to be the cyclic order of the set $L:=\{xy^{-1} \mid y\in N(x)\}$ in which $xy_1^{-1}$ is a successor of $xy_2^{-1}$ whenever the edge $xy_2$ comes immediately after the edge $xy_1$ as we move clockwise around~$x$. Note that the set $L$ is the same \fe\ vertex of $G$, and depends only on~$\SF$ and on our convention on whether to draw one or two edges per vertex for involutions. This allows us to compare spins of different vertices. Call an edge of \g {\em spin-preserving} if its two endvertices have the same spin in~$\varphi$, and call it {\em spin-reversing} if the spin of one of its endvertices is the reverse of the spin of its other endvertex. 

An embedding of a \Cg\ is called {\em \cns} if, intuitively, it embeds every vertex in a similar way in the sense that the group action carries faces to faces. Let us make this more precise.
Call a colour in $\SF$ {\em \cns} if all edges bearing that colour are  spin-preserving or all edges bearing that colour are spin-reversing in~$\varphi$. Finally, call the embedding $\varphi$ {\em \cns} if every colour is \cns\ in~$\varphi$. Note that if $\varphi$ is \cns, then there are only two types of spin in~$\varphi$, and they are the reverse of each other.



The following classical result was proved by Whitney \cite[Theorem~11]{whitney_congruent_1932} for finite graphs and by Imrich \cite{ImWhi} for infinite ones.

\begin{thm} \label{imrcb}
Let \g be a 3-connected graph embedded in the sphere. Then every automorphism of \g maps each facial path to a facial path.
\end{thm}

This implies in particular that if $\varphi$ is an embedding of the 3-connected \Cg\ \G, then the cyclic ordering of the colours of the edges around any vertex of \g is the same up to orientation. In other words, at most two spins are allowed in~$\varphi$. Moreover, if two vertices  $x,y$ of \g that are adjacent by an edge, bearing a colour $b$ say, have distinct spins, then any two vertices $x',y'$ adjacent by a $b$-edge also have distinct spins. We just proved
\begin{lem} \label{lprem}
 Let \g be a 3-connected planar \Cg. Then every embedding of \g is \cns.
\end{lem}

\Cg s of connectivity 2 do not always admit a \cnsem~\cite{DrSeSeCon}. However, in the cubic case they do; see \cite{cay2con}.

An embedding is {\em Vertex-Accumulation-Point-free}, or {\em \vapf} for short, if the images of the vertices have no accumulation point in $\R^2$.

A {\em crossing} of a path $X$ by a path or walk $Y$ in a plane graph is a subwalk $Q=e\kreis{Q}f$ of $Y$ where the end-edges $e,f$ of $Q$ are incident with $X$ on opposite sides of $X$ (but not contained in $X$) and (the image of) $Q$ is contained in $X$ (\fig{XYQ}). Note that if $Q$ is a crossing of $X$ by $Y$, then $X$ contains a crossing $Q'=g\kreis{Q}h$ of $Y$ by $X$, which we will call the \emph{dual crossing} of $Q$.

\showFig{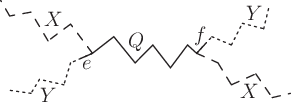}{A crossing of~$X$ by~$Y$.}

For a closed walk $W$ and $n\in\nat$, let $W^n$ be the $n$-times concatenation of~$W$.
Two closed walks $R$ and $W$ \emph{cross} if there are $i,j\in\nat$ such that $R^i$ contains a crossing of a subwalk of~$W^j$.
They are \emph{nested} if they do not cross.

\subsection{Fundamental groups of planar graphs}

Let $G$ be a graph.
The \emph{sum} of two walks $W_1,W_2$ where $W_1$ ends at the starting vertex of~$W_2$ is their concatenation.
Let $W=x_1x_2\ldots x_n$ be a walk.
Its \emph{inverse} is $x_n\ldots x_1$.
If $x_{i-1}=x_{i+1}$ for some~$i$, we call the walk $W':=x_1\ldots x_{i-1}x_{i+2}\ldots x_n$ a \emph{reduction} of~$W$.
Conversely, we \emph{add} the \emph{spike} $x_{i-1}x_ix_{i+1}$ to~$W'$ to obtain~$W$.
If $W$ is a closed walk, we call $x_i\ldots x_nx_1\ldots x_{i-1}$ a \emph{rotation} of~$W$.

Let $\VF$ be a set of closed walks.
The smallest set $\overline{\VF}\supseteq \VF$ of closed walks that is invariant under taking sums, reductions and rotations and under adding spikes is the set of closed walks \emph{generated by $\VF$}.
We also say that any $V\in\overline{\VF}$ is \emph{generated by~$\VF$}.
A closed walk is \emph{indecomposable} if it is not generated by closed walks of strictly smaller length.
Note that no indecomposable closed walk $W$ has a \emph{shortcut}, i.\,e.\ a (possibly trivial) path between any two of its vertices that has smaller length than any subwalk of any rotation of~$W$ between them.
In particular, indecomposable closed walks induce cycles.

For any $\eta$ in~$\pi_1(G)$, the fundamental group of~$G$, let $W_\eta\in\eta$ be the unique reduced closed walk in~$\eta$ and $W_\eta^\circ$ be its (unique) cyclical reduction.
For $\VF\sub\pi_1(G)$, set
\[
\VF^\circ:=\{W_\eta^\circ\mid\eta\in\VF\}.
\]
By $\WF(G)$ we denote the set of all closed walks in~$G$.

The following theorem is an immediate consequence of~\cite[Theorem 6.2]{planarCycles2}, which is a generalisation of the main theorem of~\cite{planarCycles}.

\begin{thm}\label{thm_fundgr}
Let $G$ be a planar locally finite $3$-connected graph and $\Gamma$ a group acting on~$G$.
Then $\pi_1(G)$ has a generating set $\VF$ such that $\VF^\circ$ is a $\Gamma$-invariant nested generating set for $\WF(G)$ that consists of indecomposable closed walks.
\end{thm}

\section{3-connected planar \Cg s admit \splpr s}\label{3con}


\subsection{\Plpr s --- the special case} \label{secPlpr}
We now give the crucial definition of our paper, that of a (special) \plpr. The intuition behind it comes from the notion of a \cnsem\ given above: a \plpr\ is a group presentation endowed with some additional data (forming what we will call an \empr) which, once we have proven planarity of the corresponding \Cg\ \G, will describe the local structure of a \cnsem\ of \G, that is, the spin and the information of which generators preserve or reflect it. 

\medskip
Given a group presentation $\PF=\left<\SF \mid \RF\right>$, 
we will distinguish between two types of generators $s$: those for which we have $s^2$ as a relator in $\RF$ and the rest. 
The reasons for this distinction will become clear later.
Generators $t$ for which the relation $t^2$ is provable but not explicitly part of the presentation might exist, but do not cause us any concerns. Given a group presentation $\PF=\left<\SF \mid \RF\right>$, we thus let $\IF=\IF(\PF)$ denote the set of elements $s\in \SF$ such that $\RF$ contains the relator $s^2$ or $s^{-2}$. 

Let $\SF'= \SF \cup (\SF \sm \IF)^{-1}$. For example, if $\PF=\left<a,b,c \mid a^2, b^2\right>$, then $\SF'= \{a,b,c,c^{-1}\}$.

A  {\em spin} on $\PF=\left<\SF \mid \RF\right>$ is a cyclic ordering of $\SF'$ (to be thought of as the cycling ordering of the edges that we expect to see around each vertex of our \Cg\ once we have proved that it is planar).

An {\em \empr} is a triple $(\PF, \sigma, \tau)$ where  $\PF=\left<\SF \mid \RF\right>$ is a group presentation, $\sigma$ is a spin on $\PF$, and $\tau$ is a function from $\SF$ to $\{0,1\}$ (encoding the information of whether each generator is \spre\ or \srev).

To every \empr\ $\PF, \sigma, \tau$ we can associate a tree $\ttt$ with an \vapf\ embedding in $\R^2$.
As a graph, we let $\ttt$ be $Cay\left<\SF \mid s^2, s \in \IF\right>$, and let $o=o_\ttt$ denote the identity of \ttt\ seen as a \Cg. Easily, we can embed $\ttt$ in $\R^2$ in such a way that for every vertex $v$ of $\ttt$, one of the two cyclic orderings of the colours of the edges of $v$ inherited by the embedding coincides with $\sigma$ and moreover, for every two adjacent vertices $v,w$ of $\ttt$, the clockwise cyclic ordering of the colours  of the edges of $v$ coincides with that of $w$ if and only if $\tau(a)=0$ where $a$ is the colour of the $v$--$w$ edge. (If $\tau(a)=1$, then the clockwise ordering of $v$ coincides with the anti-clockwise ordering of $w$.)

Given a word $W$, we let $W^\infty$ be the 2-way infinite word obtained by concatenating infinitely many copies of $W$.
We say that two words $W,Z\in \RF$ {\em cross}, if there is a 2-way infinite path $R$ of $\ttt$ induced by  $W^\infty$ and a 2-way infinite path $L$ induced by $Z^\infty$ such that $L$ meets both components of $\R^2 \sm R$.
Note that, if two non-trivial words form closed walks in the \Cg, then the words cross if and only if the closed walks cross.

For example, consider the presentation $\PF=\left<n,e,s,w \mid n^2, e^2, s^2, w^2\right>$, the spin $n,e,s,w$ (read `north, east, south, west'), and $\tau$ identically 0. Then any word containing $ns$ as a subword crosses any word containing $ew$. The word $nesw$ however crosses no other word, and indeed adding that word to the above presentation yields a planar \Cg: the square grid.

\begin{defi} \label{defsplpr}
A  {\em (special) \plpr} is an \empr\ $(\PF, \sigma, \tau)$ such that 
\begin{enumerate}[(sP1)]
\item \label{plpri} no two relators $W,Z\in \RF$ cross, and 
\item \label{plprii} for every relator $R$, the number of occurrences of letters $s$ in $R$ with $\tau(s)=1$ (i.e.\ \srev\ letters) is even; here, the symbol $s^n$ counts as $|n|$ occurrences of $s$.
\end{enumerate}
\end{defi}

Requirement (sP\ref{plprii}) is necessary, as the spin of the starting vertex of a cycle must coincide with that of the last vertex. 

The following lemma will later allow us to assume without loss of generality that
no relator of $\PF$ is a sub-word of a rotation of another relator.

%
%

\comment{
\begin{lem} \label{auxcross}
Let $T$ be a finite tree embedded in the plane, and let $A,B,C,D,E,F$ be distinct leaves of $T$. Suppose $\bigcap aTd, cTd, bTe, bTf \neq \emptyset$. Then at least two of the $\binom{4}{2} $ pairs of paths formed from $aTd, cTd, bTe, bTf$ cross.
\end{lem}
\begin{proof}
This is an easy combinatorial exercise that can be solved by a straightforward case analysis.
\end{proof}
}


\begin{lem} \label{nosubwords}
Let $(\PF=\left< \SF, \RF \right>, \sigma, \tau)$ be a finite \splpr. Then there is a \splpr\ $(\PF'=\left< \SF, \RF' \right>, \sigma, \tau)$ such that $\PF$ and $\PF'$ yield the same \Cg, and 
no element of~$\RF'$ is a proper subword of another element of~$\RF'$.
\end{lem}

\begin{proof}

We will perform induction on the total length of the words in $\RF$.

Suppose that $\RF$ contains a word $W$ and a (rotation of a) proper superword $R=WR_2$ of~$W$. 
To begin with, we may assume that the double rays $oW^\infty$ and $oR^\infty$ induced by them do not coincide. For if this is the case, then $W=U^n$ and $R=U^m$ for some common subword $U$, and it is easy to modify $\RF$ to avoid this situation by replacing $R,W$ with an appropriate power of $U$.

We may further assume that the double rays $oW^\infty$ and $oR^\infty$ are `as close as possible' to each other in the following sense.

Let \ttt\ be the plane tree corresponding to $\PF, \sigma, \tau$ as defined before Definition~\ref{defsplpr}.
Given three double rays $P,S,T$ in \ttt\ which are pairwise non-crossing, we say that $S$ lies \emph{between} $P,T$ if $\bigcap \{P,S,T\} \neq \emptyset$ and  $P$ and $T$ lie in distinct components of $\R^2 \sm S$. 

Note that if we fix $P$ and $T$, then there can only be finitely many double rays $S$ induced by a word in $\RF$ lying {between} $P,T$ because $\RF$ is finite.

We say that the double rays $P,T$ in \ttt\ are \emph{neighbours}, if no double ray  induced by a word in $\RF$ lies between them.

Now if $\RF$ contains a word $W$ and a proper superword $R=WR_2$ of $W$, we may assume that the double rays $oW^\infty$ and $oR^\infty$ induced by them are neighbours, because any periodic double ray lying between them contains the path $oW$ and is therefore induced by a superword or subword of $W$.

Our aim is to replace the word $R$ by its subword $R_2$ to obtain an equivalent presentation that is closer to satisfying our assertion than $\PF$ is. Therefor, we need to show that $R_2$ does not cross any of our relators. This will be a consequence of
\labtequ{R2}{Any word in $\RF\cup\{R_2\}$ crossing $R_2$ (with respect to the spin data $\sig,\tau$) also crosses $R$ or $W$.}

Let $a,z$ denote the first and last letter of $R$ respectively. Let $y$ denote the last letter of $W$, and $d$ the first letter of $R_2$. Note that $W$ starts with $a$ too, and $R_2$ ends with $z$. We have $a\neq y^{-1}$ because $W$ is reduced, and $a\neq z^{-1}, d\neq y^{-1}$ because $W$ is reduced. 

Furthermore, we may assume that $y\neq z$: for otherwise we can rotate both $W$ and $R$ by moving the letter $y=z$ from the end to the beginning, extending the intersection of the two words; here we used the fact that $R^\infty$ cannot coincide with $W^\infty$ as we noted above.

Thus $a, z^{-1}, y^{-1}$ are all distinct; let us assume that they appear in \sig\ in that order. Our first task is to decide the relative position of $d$ in \sig\ with respect to those letters. It is still possible that $d$ coincides with $z^{-1}$ or $a$.

Recall that \sig\ is a cyclic ordering on our letters $\SF \cup \SF'$. We use the notation $\sig(l,m)$ to denote those letters coming after $l$ and before $m$ in \sig. If we want to include $l$ or $m$ we use the notation $\sig[l,m)$ or $\sig(l,m]$.
 
If $d\in \sig(a,z^{-1})$, then $R$ would cross itself as its rotations contain both $za$ and $yd$ as subwords, which is impossible by~(sP\ref{plpri}).

If $d\in \sig(y^{-1},a)$, then $R$ would cross $W$ 
as can be seen in \fig{regions} by observing the double ray whose two ends are marked $o(R_2 W)^{+\infty}$, $o(R_2 W)^{-\infty}$, which double ray is induced by $R$; here we used the fact that the vertex $w$ at which the path $oW$ ends  has the same spin as $o$ in the embedding of \ttt\ we used to define crossings, because $W$ satisfies~(sP\ref{plprii}).

If $d=a$, then we can apply one of the two above arguments to the first vertex at which the rays $oR^{\infty}$ and $oW^{\infty}$ split to prove that $R$ crosses either itself or $W$, which is again a contradiction.

These facts combined prove that $d\in \sig[z^{-1},y^{-1})$. We have now gathered enough information about \sig\ to allow us to prove \eqref{R2}. 

Suppose that $R_2$ crosses some word $X$, i.e.\ there are two crossing double rays $P,T$ where $P$ is induced by $R_2$ and $T$ is induced by  $X$. 

Let $Q = e\kreis{Q}f \subseteq T$ be a crossing of $P$ by $T$ as defined in \Sr{secDem} (\fig{XYQ}), and let $Q' = g\kreis{Q}h\subseteq P$ be its dual crossing.


Let us assume first that $d\neq z^{-1}$; the case $d= z^{-1}$ will be similar.

We can translate $P$ and/or $T$ by some automorphism of \ttt\ so as to ensure that the edge $e$ is incident with the path $oR_2$ induced by $R_2$ at $o$ (\fig{regions}), and in fact $e$ is not incident with the last vertex of that path (but may be incident with $o$). 

\medskip
\begin{figure}[htbp]
\begin{overpic}[width=.95\linewidth]{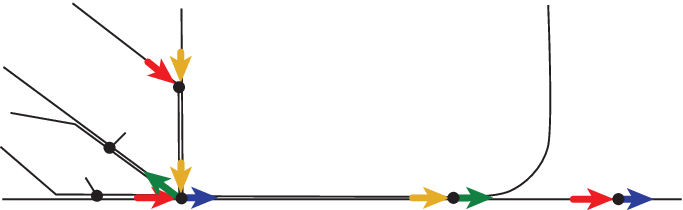}
\put(25,-2){$o$}
\put(60,-2){$w:=oW$}
\put(60,-2){$w:=oW$}
\put(88,-2){$oR$}
\put(29,-2){$a$}
\put(22,-2){$z$}
\put(28,4){$y$}
\put(22,6){$d$}
\put(-4,1){$oR^{-\infty}$}
\put(101,1){$oR^{+\infty}$}
\put(24,30){$oW^{-\infty}$}
\put(78,31){$oW^{+\infty}$}
\put(13,5){$f$}
\put(18,12){$e$}
\put(-5,8.5){$oR_2^{-\infty}$}
\put(-5,14){$oR_2^{+\infty}$}
\put(4,31){$o(R_2 W)^{-\infty}$}
\put(-3,22){$o(R_2 W)^{+\infty}$}
\put(8,8){$A_1$}
\put(15,18){$A_2$}
\put(50,15){$A_3$}
\put(50,-3){$A_4$}
\put(21,27){$A_5$}
\put(85,10){$A_6$}
\end{overpic}

   \caption{\small The four double rays $oR^{\infty}, o(R_2 W)^{\infty}, oR_2^{\infty}$ and $ oW^{\infty}$ and the regions $A_i$ they define. Here, the notation $oR^{+\infty}$ means the ray starting at $o$ induced by the 1-way infinite word obtained by repeating $R$; $oR^{-\infty}$ is defined similarly by repeating the word $R^{-1}$ instead.} \label{regions}

\end{figure}
Let us fix the embedding $\rho$ of \ttt\ in the plane complying with $\sig,\tau$ as described before Definition~\ref{defsplpr}. We use the four double rays $oR^{\infty}, o(R_2 W)^{\infty}, oR_2^{\infty}$ and $ oW^{\infty}$ to divide the plane into regions $A_i$ as shown in \fig{regions}. Then the edge $e$ must lie in one of these regions, and in each case we obtain a contradiction as follows. 

To begin with, note that exactly one of $e,f$ lies in $A_1$, and we may assume without loss of generality that $f$ does, so $e$ does not lie in $A_1$.

{\bf Case 1}: If $e$ lies in $A_2$, then $T$ crosses not only $P$ but also the double ray $o(R_2 W)^{\infty}$. But since $R_2 W$ is a rotation of the word $R= W R_2$, \eqref{R2} is proved in this case. 

{\bf Case 2}: If $e$ lies in $A_3$, (which can only happen if $e$ is incident with $o$), then $T$ crosses $ oW^{\infty}$ and again \eqref{R2} holds. 

{\bf Case 3}: If $e$ lies in $A_4$, then $T$ crosses $ oR^{\infty}$, and  again \eqref{R2} is proved. 

{\bf Case 4}: If $e= oa$, then $X\neq R_2$ as $oR_2^\infty$ does not contain the $oa$.
Furthermore, as $T$ cannot enter the regions $A_3, A_4$ for the aforementioned reasons, it has to contain the whole path $oW$ and then continue in the closure of the region $A_6$. But then some rotation of the relator $X$ inducing $T$ is either a subword or a superword of $W$, and moreover $T$ lies between
 $ oW^{\infty}$ and $oR^{\infty}$. This however contradicts our choice of $R,W$ to be neighbours.

{\bf Case 5}: If $e= oy^{-1}$, then we again have $X\neq R_2$ as $oy^{-1}$ does not lie on $oR^\infty$.
We apply a similar argument as before: as $T$ cannot enter the regions $A_2, A_3$, it has to contain the whole path $oW^{-1}$ and then continue in the closure of the region $A_5$. Then  
$T$ lies between $ oW^{\infty}$ and $o(R_2 W)^{\infty}$. As the latter can be induced by $R$, this again contradicts our choice of $R,W$ as neighbours.

These contradictions complete the proof of \eqref{R2} in the case where $d\neq z^{-1}$. The case $d= z^{-1}$ is very similar: the only difference is that the region $A_1$ is a bit smaller in \fig{regions}.

\medskip
We now claim that $R_2$ crosses none of the words in $\RF \cup \{R_2\}$. 

Indeed, if $R_2$ crosses a word in $\RF$, then applying \eqref{R2} to that word we obtain a contradiction to the fact that $\PF$ was a \splpr.

If $R_2$ crosses itself, then applying \eqref{R2} again we deduce that $R_2$ crosses $R$ or~$W$, which are elements of $\RF$. But then we are in the previous situation, which cannot occur (here we used the fact that crossing is a symmetric relation).

Keeping the spin data $\sigma, \tau$, it is clear that $R_2$ satisfies (sP\ref{plprii}) as both $W$ and $R$ satisfied that property. 

Since $\PF$ satisfied (sP\ref{plpri}), and we have just proved that $R_2$ crosses neither itself nor any other relator in $\RF$, this means that the presentation obtained from $\PF$ by adding $R_2$ as a relator is still a \splpr. Note that the relator $R= WR_2$ now becomes redundant, and we can remove it. Hence we obtain a \splpr\ of the same \Cg\ in which the relators have strictly smaller total length.

We can repeat this for as long as there are relators in our presentation that are a subword of each other. Since the total length decreases in each step, the process terminates after finitely many steps yielding the desired presentation.
\end{proof}

\subsection{Existence of \splpr s}
We now prove that planar $3$-connected \Cg s admit \splpr s.

\begin{thm}\label{thm_3con}
Every planar, locally finite, $3$-connected \Cg\ admits a \splpr.
\end{thm}

\begin{proof}
Let $G$ be a planar, locally finite, $3$-connected \Cg, and let $\Gamma:=\Gamma(G)$ be its group.
By Droms~\cite[Theorem~5.1]{DroInf}, $\Gamma$ admits a finite presentation $\PF=\left<\SF\mid\RF\right>$.
We may replace the generators~\SF\ by those finitely many generators that we used to obtain the \Cg~$G$, that is, we may assume that \SF\ was used to obtain~$G$.
Let $\DF$ be a generating set of~$\pi_1(G)$ such that $\DF^\circ$ is a nested $\Gamma$-invariant set of indecomposable closed walks in~$G$.
This exists by Theorem~\ref{thm_fundgr}.
Let $\DF'\sub\DF^\circ$ be the subset of those elements of~$\DF^\circ$ that contain the vertex~$o$.
Then the set $\RF_{\DF'}$ of words corresponding to the elements of~$\DF'$ yields a presentation $\PF'=\left<\SF\mid\RF_{\DF'}\right>$ of~$\Gamma$ since $\DF$ generates $\pi_1(G)$.
Note that a priori the set $\DF'$, and hence also $\RF_{\DF'}$, might be infinite.
As $\Gamma$ is finitely presented, it is well-known, see e.\,g.~\cite{Baumslag-Topics}, that we can use Tietze-transformations to obtain a finite subset $\RF'$ of~$\RF_{\DF'}$ such that $\left<\SF\mid\RF'\right>$ is a finite presentation of~$\Gamma$.
We claim that $(\left<\SF\mid\RF'\right>,\sigma,\tau)$ is a planar presentation, where $\sigma$ is the spin of~$1$ in some embedding $\rho$ of~$G$ in $\R^2$, and $\tau(s)$ is 0 for those $s\in \SF$ \st\ the spin of $1$ coincides with the spin of the vertex $s$ of $G$ in $\rho$ (and $\tau(s)$ is 1 for every other $s\in \SF$).

Indeed, it is easy to check using the nestedness of the finitely many closed walks that correspond to the relators~$\RF'$ and \Lr{lprem} that no two elements of $\RF'$ cross.
\end{proof}


Recall the notion of almost planarity introduced in Section~\ref{subsec_results}: a Cayley complex~$X$ of a presentation $\left<\SF\mid\RF\right>$ is \emph{almost planar} if there is a mapping $\rho\colon X\to\real^2$ such that $\rho$ is injective on the 1-skeleton of~$X$, and
for every two $2$-simplices of~$X$, the images of their interiors under $\rho$ are either disjoint or one of these images is contained in the other. (Here, we are using our convention that elements $s$ of $\SF$ such that $s^2$ is a relator in $\RF$ give rise to single, undirected edges in $X$.)
\Tr{thm_3con} has the following consequence, which was conjectured in \cite{cayley3}.

\begin{cor}\label{cor_3con}
Every planar, locally finite, $3$-connected planar \Cg~$G$ is the $1$-skeleton of an almost planar Cayley complex of the group~$\Gamma(G)$ of~\G.
\end{cor}
\begin{proof}
\medskip
Since $G$ is planar, there is an embedding $\rho': G\to \R^2$ by definition. We will extend $\rho'$ to the desired map $\rho$ from the  Cayley complex $X$ of $\Gamma(G)$ with respect to the presentation $\left<\SF\mid\RF'\right>$ from above. For this, given any 2-cell $Y$ of $X$ with boundary cycle $C$, we embed $Y$ in the finite component of $\R^2 \sm C$. It is a straightforward consequence of the nestedness of~$\DF$ that the resulting map $\rho$ has the desired property.
\end{proof}

\section{\Plpr s yield planar \Cg s: the consistent case} \label{secplty}

In this section we 
prove that every \splpr\ ---as defined in \Sr{secPlpr}--- defines a planar \Cg\ with a consistent embedding (Theorem~\ref{thmplanar}). This proof contains the fundamental arguments of this paper.
 
\subsection{Fundamental domains}

Let again $\PF= \left<\SF \mid \RF\right>$ be a presentation, and $G:=Cay(\PF)$ its \Cg. Let $T_\SF= Cay\left<\SF \mid \emptyset\right>$ be the corresponding free tree. Let $N(\RF)$ denote the normal closure of $\RF$ in the group of $\left<\SF \mid \emptyset\right>$, and note that $N(\RF)$ acts by automorphisms on $T_\SF$. Then $G$ is, almost by definition, the quotient  $T_\SF/ N(\RF)$ with respect to that action.

In this section we consider all graphs as 1-complexes. The following lemma is folklore; we include a proof for the convenience of the reader.

\begin{lem}\label{FD}
$T_\SF$ admits a connected fundamental domain for the action of $N(\RF)$.
\end{lem}
\begin{proof}
Let $D$ be a maximal subgraph of $T_\SF$ that is connected and meets each $N(\RF)$-orbit in at most one point; such a $D$ exists by Zorn's lemma. We claim that $D$ meets every $N(\RF)$-orbit. For if not, then there exist two adjacent vertices $x,xs$ in $T_\SF$ (where $s\in \SF$) such that $xs$ does not belong to any orbit represented by $D$ but $x$ does. Let $x'= Rx$, where $R \in N(\RF)$, be the vertex of $D$ that lies in the same orbit as $x$. Then the vertex $Rxs$ is connected to $x'$, and its orbit is not represented by $D$. This contradicts the maximality of $D$, since $D \cup Rxs$ is connected.
\end{proof}

An {\em open star} is a subspace of a graph consisting of a single vertex and all open half-edges incident with it. A {\em star} is the union of an open star with some of the midpoints in its closure.

For the connected fundamental domain $D$ provided by \Lr{FD}, we may assume without loss of generality that 
\labtequ{stars}{$D$ is a union of stars,} 
since the action of~$N(\RF)$ never identifies two points in the same star.

\subsection{Proof of planarity} \label{secpltyproof}

In this section we prove 
\begin{thm} \label{thmplanar}
If $(\PF, \sigma, \tau)$ is a special \plpr\ with countably many generators and relators, then its \Cg\ $Cay(\PF)$ is planar. Moreover, it admits a consistent embedding, with spin $\sigma$ and spin-behaviour of generators given by $\tau$.
\end{thm}

For the rest of this section, let us fix $\PF=\left<\SF \mid \RF\right>$ as above, and let $G:= Cay(\PF)$.
Recall the definition of the embedded tree \ttt\ from \Sr{secPlpr}, i.\,e.\ $\ttt=Cay\left<\SF\mid s^2, s\in\IF\right>$.

Note that it suffices to prove the statement for a finite presentation $\PF$; the countably infinite case can then be deduced as follows. If \g does not admit a consistent embedding, then by a standard compactness argument  there is a finite subgraph $H \subset G$ that does not admit a consistent embedding. It is an easy exercise to show that for some finite $\SF'\subset \SF$ and $\RF'\subset \RF$, the \Cg\ $Cay(\left<\SF',\RF'\right>)$ also contains $H$ as a subgraph. 
But $\left<\SF',\RF'\right>$ is a finite presentation which is planar with respect to the restriction of $\sigma, \tau$ to $\SF',\RF'$, leading to a contradiction.

\medskip

Let $D$ be a connected fundamental domain of $\ttt$ with respect to the action of $N(\RF)$, provided by \Lr{FD}. Recall that we may assume that $D$ is a union of stars. Thus the closure $\cls{D}$ of $D$ in $\ttt$ is the union of $D$ with all midpoints of edges that have exactly one half-edge in $D$. Moreover, $G$ can be obtained from $\cls{D}$ by identifying pairs of such midpoints: each midpoint $m$ in $\cls{D}\sm D$ is identified with the unique midpoint $m'$ in $D$ that is {\em $N(\RF)$-equivalent} to $m$, where we call two points or subsets $X,Y$ of \ttt\  $N(\RF)$-equivalent when they lie in the same orbit of the action of $N(\RF)$ on~\ttt. 
Note that $m'$ might coincide with $m$, which is the case exactly when it lies on an edge coloured by a generator in the set $\IF$ of explicit involutions.

We claim that
\labtequ{rspre}{every two $N(\RF)$-equivalent vertices of \ttt\ have the same spin.} 
Indeed, this follows from condition (sP\ref{plprii}) of the definition of a \plpr, according to which every element of $\RF$ joins vertices with same spins.

\medskip
To show that $G$ is planar, and it even admits a consistent embedding, it will suffice to show that these pairs of identified points are {\em nested} in the embedding of $\cls{D}$ inherited from the embedding of $\ttt$. Here, we say that two pairs of midpoints  $x,x'$ and $y,y'$ in $\cls{D}\sm D$ are nested, if the $x$--$x'$ path in $\cls{D}$ does not cross the $y$--$y'$ path, where we define \emph{crossing} similarly to \Sr{secDem}. 

Assuming that such pairs of points are nested, it is easy to 
 prove that \g is planar: note that we can cut a closed domain $D'$ of $\R^2$ homeomorphic to a closed disc such that $D'\cap \ttt=\cls{D}$. Let $D''$ be a homeomorphic copy of $D'$, and glue $D'$ to $D''$ by identifying all pairs of corresponding points of their boundaries to obtain a homeomorph $S$ of the sphere. For every pair $x,x'$ of $N(\RF)$-equivalent points of $\cls{D}$, let $X$ be the $x$--$x'$ path in $\cls{D}$ and let $X''$ be its copy in $D''$. Nestedness implies by definition that these arcs $X''$ can be continuously deformed into  pairwise disjoint arcs. Therefore, the union of $\ttt\cap \cls{D}$ with all these arcs is an embedding of \g on the sphere $S$ (where every midpoint of an edge in $\cls{D}$ became a closed arc). This embedding is consistent because the embedding of \ttt\ we started with is. 

\medskip
It thus only remains to prove this nestedness. Our intuition for this is that when both the $x$--$x'$ path and the $y$--$y'$ path from above are induced by relators, then these paths cannot cross since that would imply that the corresponding relators cross, which is forbidden. In the next section we will extend this idea to arbitrary pairs of such points, using the fact that the aforementioned paths are cycles of \g and cycles of \g can be `proved' using relators.

\subsection{Nestedness in $\cls{D}$} \label{secnest}
Let $\pi$ denote the canonical covering map from \ttt\ to \G, and let $o_G:=\pi(o_\ttt)$ denote the identity element of \G.
Given a  relator $W$, let $W_o$ denote the closed walk $o_G W$ in \g induced by $W$ at $o_G$. Let $\ttt_W:=\pi^{-1}(W_o)$, and note that $\ttt_W$ is a union of a set of double-rays of $\ttt$, which set we denote by $\ttt[W_o]$,
and along each such double-ray we can read the 2-way infinite word $W^\infty$ obtained by repeating $W$ indefinitely.
Indeed, the only case that could prevent $W^\infty$ from 
spanning a double-ray is when $W=b$ for some $b$ for which $b^2$ is a relator, but then $W$ would be a subword of the relator $b^2$.
However, applying \Lr{nosubwords} we may assume that this is not the case. 
Equivalently, $\ttt_W$ is the union of all double rays $N(\RF)$-equivalent with the double-ray induced by $W$ at $o_\ttt$. 

We start our proof of nestedness by noting that in a plane graph, every cycle $C$ separates the graph into two (possibly empty) sides $I,O$ by the Jordan curve theorem, with no edges of the graph joining $I$ to $O$. Although we have not yet embedded $G$ in the plane, we will be able to show that cycles, or closed walks, of \g that are induced by relators enjoy a similar property by exploiting the embedding of \ttt\ and the non-crossing property of relators.

\subsubsection{Bipartitioning the faces of \ttt} \label{subBipT}

The {\em dual} graph $\ttt^*$ of \ttt\ is the graph whose vertex set is the set of faces of \ttt, and two faces of \ttt\ are joined with an edge $e^*$ of $\ttt^*$ whenever their boundaries share an edge $e$ of $\ttt$.
Given two faces $F,H$ of \ttt, and an $F$--$H$~path $P_{FH}$ in $\ttt^*$, we let $Cr(\ttt[W_o],P_{FH})$ denote the {\em number of crossings} of $\ttt[W_o]$ by $P_{FH}$; to make this more precise, for a double-ray $T$ in $\ttt[W_o]$, we write $cr(T, P_{FH})$ for the number of edges $e$ in $T$ such that $P_{FH}$ contains $e^*$, and we let  $Cr(\ttt[W_o],P_{FH}):= \sum_{T\in \ttt[W_o]} cr(T,P_{FH})$. We claim that
\labtequ{Cr}{for every two faces $F,H$ of \ttt, the parity of the number of crossings $Cr(\ttt[W_o],P_{FH})$ is independent of the choice of the path $P_{FH}$.}
To see this, note that if $C$ is a cycle in $\ttt^*$, then $Cr(\ttt[W_o],C)$ ---defined similarly to $Cr(\ttt[W_o],P_{FH})$--- is even because the embedding of \ttt\ is \vapf\ and so any ray entering the bounded side of $C$ has to exit it again. This immediately implies~\eqref{Cr}.

This fact allows us to introduce the following definition
\begin{defi}
Given two faces $F,H$ of \ttt, we write $F\sim H$ if some, and hence every, $F$--$H$~path $P_{FH}$ in $\ttt^*$ crosses $\ttt_W$ an even number of times, i.e.\ if $Cr(\ttt[W_o],P_{FH})$ is even.
\end{defi}

Note that $\sim$ is an equivalence relation of the set of faces $\FF$ of \ttt. Moreover, it uniquely determines an (unordered) bipartition $\{I,O\}$ on $\FF$ by choosing one face $F$ and letting $I:=\{H\in \FF \mid H \sim F\}$ and $O:= \FF \sm I$. Note that $\ttt_W$ decomposes $\R^2$ into regions each of which is in a single side of this bipartition, and crossing $\ttt_W$ corresponds to alternating between $I$ and $O$. It will turn out, after we prove that \g is planar, that $I,O$ are the lifts of the two sides of the closed walk $W_o$ of \G. The following lemma shows that this is plausible and will be needed later.

\begin{lem} \label{Ninv}
The relation $\sim$ is invariant under the action of $N(\RF)$ on~\ttt. 
\end{lem}
\begin{proof}

It suffices to prove that if $F,H$ are faces in the same orbit of $N(\RF)$, then $F\sim H$. We may assume that there are vertices $x,y$ in the boundaries of $F,H$ respectively, such that $y= x w R w^{-1}$ for some word $w$ and some relator $R\in \RF$: by the definition of the normal closure $N(\RF)$, if we can prove $F\sim H$ in this case, we can prove $F\sim H$ for every two  $F,H$ in the same orbit of $N(\RF)$.

Since we are free to choose any $F$--$H$~path  in $\ttt^*$ by \eqref{Cr}, let us choose $P_{FH}$ to be one that starts with an edge $e^*$ where $e$ is incident with both $x$ and $F$, finishes with an edge $f^*$ where $f$ is incident with both $y$ and $H$, and does not cross the walk from $x$ induced by $w R w^{-1}$ (\fig{wRw}). We need to check that $Cr(\ttt[W_o],P_{FH})$ is even.

\begin{figure}[htbp]
\centering \begin{overpic}[width=.6\linewidth]{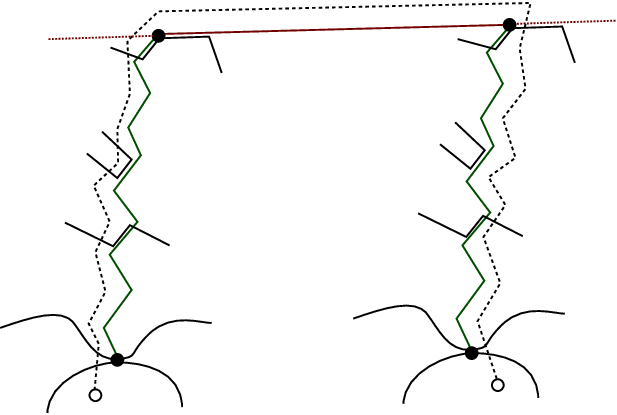} 
\put(24,35){$w$}
\put(66,35){$w^{-1}$}
\put(50,70){$P_{FH}$}
\put(50,56){$R$}
\put(18,1){$F$}
\put(73,2){$H$}
\put(9,9){$x$}
\put(67,10){$y$}
\put(25,64){$x'$}
\put(82,67){$y'$}

\end{overpic}

   \caption{\small The path $P_{FH}$ (dashed) in the proof of \Lr{Ninv} and some elements of $\ttt[W_o]$ it crosses.} \label{wRw}

\end{figure}

For this, we only need to consider those double-rays $T\in \ttt[W_o]$ for which $cr(T,P_{FH})$ is odd. We will group those $T$ into pairs, showing that their total contribution is even as desired. For simplicity, we will tacitly assume that  $x w R w^{-1}$ crosses any $T$ just once; the general case can be handled with the same arguments. 

By elementary topological arguments, any such $T$ is either crossed by our walk $x w R w^{-1}$, or it visits $x$ or $y$ and its two rays are separated by\\  $F \cup  x w R w^{-1} \cup H$. We will call the former case a crossing of type A, and the latter a crossing of type B, and we will separately show that crossings of each type come in pairs.

Recall the definition of a crossing from \Sr{secDem}.
For type A, we will distinguish the following sub-types of crossings $Q= e\kreis{Q}f\subseteq w R w^{-1}$ of such a $T$ by $w R w^{-1}$:
\begin{enumerate}
\item \label{ci} both $e,f$ lie in $R$;
\item \label{cip} one of $e,f$ lies in $w$ and the other in $w^{-1}$;
\item \label{cii} both $e,f$ lie in $w$ or they both lie in $w^{-1}$; or
\item \label{ciii} exactly one of $e,f$ lies in $R$ (and the other in either $w$ or $w^{-1}$).
\end{enumerate}
Type \ref{ci} cannot occur, as it would imply that the relator $R$ crosses with the relator of $T$, and we are assuming $\PF$ to be a \plpr. The second one is also impossible, as it would imply that $R$ is contained in a double-ray in $\ttt[W_o]$, which would in turn imply that one of $T,R$ is a subword of the other, and we are forbidding this too in a \plpr.

We will define a bijection between those crossings $Q$ that have an end-edge in $w$ and those that have an end-edge in $w^{-1}$, showing that the total number of crossings is even as desired.

For crossings of type \ref{cii}, it is easy to bijectively map each crossing $Q$ with $e,f$ in $w$ to a crossing $Q'$ with $e,f$ in $w^{-1}$: the automorphism of \ttt\ mapping the end-vertex $x'$ of $w$ to the starting vertex $y'$ of $w^{-1}$ translates $Q$ to a crossing $Q'$ as desired, as it translates $w$ to $w^{-1}$ and the element of $\ttt[W_o]$ crossed by $w$ to another element of $\ttt[W_o]$.

For crossings of type \ref{ciii} a similar argument applies, but we have to be more careful. Again, given a crossing $Q$ of an element $T$ of $\ttt[W_o]$ by $w R w^{-1}$ with $e$, say, in $w$ and $f$ in $R$, we translate it to a walk $Q'$ by the automorphism of \ttt\ mapping $x'$ to $y'$. We claim that $Q'$ is a crossing of the translate $T'$ of $T$ by $w R w^{-1}$. This fact is easier to see in \fig{wRw} than to explain with words, and it follows from the following three facts: a) the double-ray $R^\infty$ obtained by reading the word $R$ indefinitely starting from $x'$ does not cross $T$ by the definition of a \plpr; b) $T'$ does not contain all of $R$, since we are assuming that no relator is a sub-word of another relator, and c) $x'$ and $y'$ have the same spin since they are joined by a relator path $R$.

\medskip
It remains to consider crossings of type B, i.e.\ where $T$ visits $x$ or $y$ and its two rays are separated by $F \cup  x w R w^{-1} \cup H$. But for any such $T$, the automorphism of \ttt\ mapping $x$ to $y$ or the other way round maps $T$ to another element of  $\ttt[W_o]$ that is crossed by $x w R w^{-1}$ as often as $T$ is, therefore such crossings appear in pairs as well.

\medskip
Thus we have paired up all crossings, showing that $Cr(\ttt[W_o],P_{FH})$ is even.
\end{proof}

A further important property of our bipartition is that
\labtequ{FeH}{for every edge $e$ of \ttt, the two faces $F,H$ of $e$ lie in distinct elements of $\{{I},{O}\}$ if and only if $e\in \ttt_W$ and $e$ lies in an odd number of elements of $\ttt[W_o]$.}
Indeed, in this case the edge $e^*$ can be chosen as the $F-H$~path $P_{FH}$ in the definition of $\sim$, and $Cr(\ttt[W_o],P_{FH})$ is just the  number of elements of $\ttt[W_o]$ containing $e$.

\begin{rem}
We can define an equivalence relation on the vertices of \ttt\ similar to our $\sim$ for faces: write $x\sim y$ if the (unique) $x$--$y$~path $P_{xy}$ in \ttt\ crosses $\ttt[W_o]$ an even number of times. 
Results similar to \Lr{Ninv} and \eqref{FeH} extend to this relation, but it is more convenient to work with faces.
\end{rem}

\subsubsection{Bipartitioning the `faces' of \G} \label{subBipG}
\newcommand{\are}{\vec{e}}
\newcommand{\arE}{\vec{E}}

We would like to use the $N(\RF)$-invariance of the bipartition of $\FF$ we defined above (\Lr{Ninv}) to induce a bipartition on faces of \G, but we cannot talk about faces of \g before proving that it is planar. However, there is a way around this: for every face  $F$ of \ttt, glue a copy of the domain $\overline{F}\subset \R^2$ to \g by identifying each point $x$ of $\partial F$ with $\pi(x)$, where $\pi$ still denotes our covering map from \ttt. 
If $F,F'$ are equivalent face boundaries, in other words, if $\pi(\partial F)=\pi(\partial F')$, then we identify the corresponding 2-cells glued onto \G. These identifications ensure that every edge of $G$ is in the boundary of either exactly one or exactly two 2-cells (but might appear in the boundary of a 2-cell several times); $e$ is in the boundary of only  one 2-cell exactly when the two faces of any lift of $e$ to \ttt\ are $N(\RF)$-equivalent. Indeed, the 2-cells we introduced are bounded by walks corresponding to facial 2-way infinite words, and every edge is in exactly two such words. Moreover, the lifts of $e$ map those walks to exactly the boundaries of pairs of incident faces of \ttt, and such pairs are $N(\RF)$-equivalent for the various decks of the covering.
 
Let $\cells$ denote the set of these 2-cells, and let $\Gc=\G \cup \cells$ denote the 2-complex consisting of \g and these 2-cells.

\Lr{Ninv} now means that if $Z$ is a closed walk of \g induced by a relator, then $\{{I},{O}\}$ induces a bipartition $\pi[I],\pi[O]$ of $\cells$. 
Let us from now on denote this bipartition of $\cells$ by {\em $B_Z$}.
\medskip

Our next aim is to extend our construction of that bipartition to an arbitrary cycle of \G, showing that every cycle has two `sides'. 

To achieve this, given a cycle $C$ of $G$, we choose a `proof' $P$ of $C$; that is, a sequence of closed walks  $W_i, 1\leq i \leq k$, of \g induced by rotations of relators such that $C=\sum_{1\leq i\leq k} W_i$, where $\sum$ denotes addition in 
the simplicial homology sense. Such a sequence $(W_i)$ exists by \Lr{relcc}. For every $W_i$, let $I_{W_i},O_{W_i}$ denote the two sides of the bipartition $B_{W_i}$ of $\cells$ from above; it will not matter which element of $B_{W_i}$ we denote by $I_{W_i}$ and which by $O_{W_i}$, so we may make an arbitrary choice here. 

Define the bipartition $B_C(P):=\{I_C,O_C\}$ of the 2-cells $\cells$ of \Gc\ by setting $I_C:= \sydi_i I_{W_i}$ and $O_C:= \cells \sydi I_C$, where $\sydi$ denotes the symmetric difference. Note that 
$O_C= \sydi_i O_{W_i}$ when $k$ is odd and $O_C= \cells \sydi \left(\sydi_i O_{W_i}\right)$ when $k$ is even, where $k$ is the number of our $W_i$.

Once we have constructed a planar embedding of \G\ in the plane, it will turn out that $B_C(P)$ corresponds to the bipartition of the faces of \g into those lying inside/outside $C$.\footnote{It should be said however, that faces of \g do not correspond one-to-one to 2-cells of $\Gc$; see~\cite{planarPresII}.} To see why this might be true, consider the situation of \fig{BC} as an example; we imagine $I_C$ to be the set of 2-cells inside $C$ and $O_C$ the set of 2-cells outside it, or the other way round. \fig{BC} is only a help to our imagination since we have not yet proved \g to be planar. Our definition of $B_C(P)$ builds on this idea, but has to deal with the fact that we do not yet have an embedding of \G\ in the plane. 

\begin{figure}[htbp]
\centering \begin{overpic}[width=.4\linewidth]{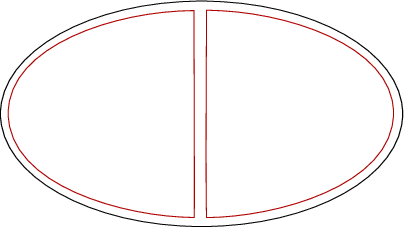} 
\put(30,43){$W_1$}
\put(57,43){$W_2$}
\put(9,53){$C$}
\end{overpic}

   \caption{\small Intuition of the definition of $B_C(P)$.} \label{BC}

\end{figure}

The reader who finds this definition surprising might be comforted to know that this is the subtlest idea of the proof of \Tr{thmplanar}, and it took us a long time to come up with. An important point, which explains it to some extent, is that $B_C(P)$ is independent of our above choices of which side to denote by $O_{W_i}$ and which by $I_{W_i}$, since a bipartition is an undirected pair of sets (the example of \fig{BC} might be helpful again). What is more important is that, as we will see below, $B_C(P)$ is independent of the choice of the proof $P$. From all we know at the moment however, it is defined as a function of the proof $P$, so let us denote it by $B_C(P)$.

Our next aim is to show that, in a certain way, $B_C(P)$ behaves like the bipartition of the faces of a plane graph induced by a cycle $C$: to move between the two sides, one has to cross an edge of $C$. This is achieved by \Lr{faces} below, for the proof of which we need the following. 

\begin{lem} \label{We}
Let $e$ be a directed edge of \G, let $W\in \RF$ be a relator which is not of the form $b^2=1$ for $b\in \SF$, and let $o_G W$ be the closed walk of \g rooted at $o_G$ induced by $W$. Then the number of double-rays in $\ttt[W_o]$ containing $e$ equals the number of times that $o_G W$ traverses $\pi(e)$. 
\end{lem}
\begin{proof}
If $o_G W$ does not traverse $\pi(e)$ then $\ttt[W_o]$ avoids $e$ and we are done. So suppose that $o_G W$ does traverse $\pi(e)$.
Let $o_G W^\infty$ denote the two-way infinite walk on \g obtained by repeating $o_G W$ indefinitely. Let $T\in \ttt[W_o]$ be the lift of $o_G W^\infty$ to \ttt\ (via $\pi^{-1}$) sending $\pi(e)$ to $e$, and note that $T$ is a double-ray containing $e$. Let $Q$ be the subpath of $T$ that starts with $e$ and finishes when a rotation of the word $W$ is completed. By the definition of $\ttt[W_o]$, there is a 1--1 correspondence between the elements of $\ttt[W_o]$ containing $e$ and the directed edges $e'$ in $Q$ that are $N(\RF)$-equivalent to $e$: each such element of $\ttt[W_o]$ can be obtained by translating $T$ by the automorphism of \ttt\ sending $e'$ to $e$. 

Now note that $o_G W$ traverses $\pi(e)$ whenever its lift $T$ traverses one of those~$e'$. Combined with the above observations this proves our assertion.
\end{proof}

\begin{lem} \label{faces}
For every $e\in E(G)$,  the bipartition  $B_C(P)$ separates 2-cells of $e$ if and only if $e\in C$.
\end{lem}
\begin{proof}
Let $I,O$ be the two elements of $B_C(P)$ as defined above. Then, letting $1_{F\in I}$ denote the indicator function of $F\in I$, we have  
$$1_{F\in I} = N_F:= | \{W_i \mid F\in I_{W_i}\} | \pmod 2,$$ 
and similarly 
$$1_{H\in I} = N_H:= | \{ W_i \mid H\in I_{W_i}\} | \pmod 2.$$
But 
$$N_F+N_H = | \{ W_i \mid W_i \text{ separates } F \text{ from } H\} | \pmod 2 $$ by the construction of $I,O$. We claim that $| \{ W_i \mid W_i \text{ separates } F \text{ from } H\} |$ is odd if and only if $e\in E(C)$. Indeed, $B_{W_i}$ separates $F$ from $H$ exactly when $W_i$ traverses $e$ an odd number of times by \eqref{FeH} and \Lr{We}, 
and $e$ is in $C$ exactly when there is an odd number of $W_i$ that traverse $e$ an odd number of times. 

Since that number is even if $e\not\in E(C)$ and odd otherwise, our last congruence yields $N_F+N_H = 1 \pmod 2$ if and only if $e\in E(C)$. Therefore, the previous congruences imply that $1_{F\in I}= 1_{H\in I}$ if  $e\not\in E(C)$ and $1_{F\in I}\neq 1_{H\in I}$ if  $e\in E(C)$, which is our claim.
\end{proof}

\Lr{faces} implies in particular that $B_C(P)$ is characterised by $C$ alone and is therefore independent of $P$, since $\Gc$ was defined without reference to $P$. Thus we can denote it by just $B_C$ from now on. In the following, we use again the definition of a crossing from \Sr{secDem}.

\begin{lem} \label{crossep}
Let $C'$ be a finite path  of \ttt\ such that $C:=\pi(C')$ is a cycle of~\G, and let $Q=e\kreis{Q}f$ be a crossing of $C'$ in \ttt. Then $B_C$ separates the 2-cells incident with $\pi(e)$ from the 2-cells incident with $\pi(f)$. Moreover, if $Q_2$ is a path of \ttt\ such that $\pi(Q_2)$ is a cycle of \G, then $Q_2$ crosses $C'$ an even number of times.
\end{lem}
\begin{proof}
Let $F$ be a face incident with the first edge $e$ of $Q$, and let $H$ be a face incident with the last edge $f$ of $Q$.
By the definition of a crossing, we can find a finite sequence $(F=)F_1, \ldots, F_k(=H)$ of faces of \ttt\ such that each $F_i$ shares an edge $e_i$ with $F_{i+1}$ and exactly one of the $e_i$ lies in $C'$: we can visit all faces incident with $Q$ until we reach~$H$. By \Lr{faces} and \Lr{Ninv}, $B_C$ separates $\pi(F_1)$  from $\pi(F_k)$. 
This proves our first assertion.
\medskip

For the second assertion, note that $\pi(Q_2)$ can be written as a concatenation of subarcs $C_1 D_1 C_2 D_2 \ldots C_k=C_1$ where each $C_i$ lifts to a crossing of $C'$ by $Q_2$ and each $D_i$ avoids $C$ and shares exactly one end-edge with each of $C_{i}$ and $C_{i+1}$. 
We proved above that the 2-cells incident with end-edges of each $C_i$ are separated by $B_C$. The same arguments imply that  the 2-cells incident with end-edges of each $D_i$ are {\em not} separated by $B_C$. Since $\pi(Q_2)$ is a cycle, this implies that $Q_2$ crosses $C'$ an even number of times. 
\end{proof}

We can now prove that all pairs of identified points of $\cls{D}$ are nested, completing the proof of \Tr{thmplanar} started at the beginning of \Sr{secplty}.
Suppose, to the contrary, there are two pairs $x,x'$ and $y,y'$ that are not nested. Let $X$ be the $x$--$x'$ path in $\cls{D}$, and let $Y$ be the $y$--$y'$ path. Then $X$ crosses $Y$ exactly once since $\cls{D}$ is a tree, contradicting the last statement of \Lr{crossep} because $\pi(X),\pi(Y)$ are cycles of \g by the definition of $\cls{D}$. Thus all such pairs are nested, and our proof is complete.

\bibliographystyle{plain}
\bibliography{collective}

\end{document}